\documentclass[11pt]{amsart}
\usepackage{amsmath,amsthm,amssymb,url}
\usepackage[latin1]{inputenc}
\usepackage[all]{xy}

\DeclareMathOperator{\Spec}{Spec}
\DeclareMathOperator{\Hom}{Hom}
\DeclareMathOperator{\Id}{Id}

\DeclareMathOperator{\Def}{Def}
\DeclareMathOperator{\End}{End}
\DeclareMathOperator{\coker}{coker}
\DeclareMathOperator{\MC}{MC}

\DeclareMathOperator{\imm}{Im}

\newcommand{\C}{\mathbb{C}}
\newcommand{\N}{\mathbb{N}}

\newcommand{\M}{\mathcal{M}}

\newcommand{\A}{\mathcal{A}}
\newcommand{\deltabar}{\bar{\partial}}
\newcommand{\Eps}{\mathcal{E}}
\newcommand{\F}{\mathcal{F}}
\newcommand{\G}{\mathcal{G}}
\renewcommand{\H}{\mathcal{H}}
\newcommand{\U}{\mathcal{U}}
\newcommand{\Oh}{\mathcal{O}}

\newtheorem{theorem}{Theorem}[section] 
\newtheorem{lemma}[theorem]{Lemma}
\newtheorem{cor}[theorem]{Corollary}
\newtheorem{prop}[theorem]{Proposition}

\newtheorem{definition}[theorem]{Definition}
\theoremstyle{remark}
\newtheorem{remark}[theorem]{Remark}

\begin{document}

\title{Local structure of Brill-Noether strata \\ in the moduli space of flat stable bundles}
\date{}

\author{Elena Martinengo}
\email{martinengo@mat.uniroma1.it}
\urladdr{www.mat.uniroma1.it/dottorato/}

\begin{abstract}
We study the Brill-Noether stratification of the coarse moduli space of locally free stable and flat sheaves of a compact K\"{a}hler manifold, proving that these strata have quadratic algebraic singularities.  
\end{abstract}


\maketitle

\section{Introduction}
Let $X$ be a compact complex K\"{a}hler manifold of dimension $n$. Let $\M$ be the moduli space of locally free sheaves of $\Oh_X$-modules on $X$ which are stable and flat.
It is known that a coarse moduli space of these sheaves can be constructed and that it is a complex analytic space (see \cite{Norton} or \cite{Le Poitier} for the algebraic case). Moreover the following result, proved in \cite{Nadel} and \cite{Goldman-Millson 1}, determines the type of singularities of this moduli space:  
\begin{theorem} \label{teo per M}
The moduli space $\M$ has quadratic algebraic singularities.
\end{theorem}
In his article \cite{Nadel}, Nadel constructs explicitly the Kuranishi family of deformations of a stable and flat locally free sheaf of $\Oh_X$-modules $\Eps$ on $X$ and he proves that the base space of this family has quadratic algebraic singularities. 
Whereas the proof given by Goldman and Millson in \cite{Goldman-Millson 1} is based on the study of a germ of analytic space which prorepresents the functor of infinitesimal deformations of a sheaf $\Eps$. They find out this analytic germ and prove that it has quadratic algebraic singularities. \\

This paper is devoted to the local study of the strata of the Brill-Noether stratification of the moduli space $\M$ and in particular, in the same spirit as Theorem \ref{teo per M}, to the determination of their type of singularities. 

In section \ref{type of singularities} we study an equivalence relation between germs of analytic spaces under which they are said to have the same type of singularities. We  prove that this relation is formal (Proposition \ref{prop formalita}) and that the set of germs with quadratic algebraic singularities is closed under this relation (Theorem \ref{teo sing quadr}).

In section \ref{def di N} we introduce the Brill-Noether stratification of the moduli space $\M$, we are interested in.
The subsets of this stratification are defined in the following way. We fix integers $h_i\in \mathbb{N}$, for all $i= 0\ldots n$, and we consider the subspace $\mathcal N (h_0\ldots h_n) \subset \M$ of stable and flat locally free sheaves of $\Oh_X$-modules on $X$, with cohomology spaces dimensions  fixed: $\dim H^i=h_i$, for all $i= 0\ldots n$. 
Our aim is to study the local structure of these strata $\mathcal N (h_0\ldots h_n)$, proving the following 
\begin{theorem}[Main Theorem]
The Brill-Noether strata $\mathcal N(h_0\ldots h_n)$ have quadratic algebraic singularities.
\end{theorem} 

In sections \ref{def di N} and \ref{F con DGLA},
we define and study the functor $\Def_{\Eps}^0$ of infinitesimal deformations of a stable and flat locally free sheaf of $\Oh_X$-modules $\Eps$ on $X$, such that $H^i(X,\Eps)=h_i$ for all $i=0\ldots n$, which preserve the dimensions of cohomology spaces. 

In section \ref{proof of the main theorem}, we find out an other functor linked to $\Def_{\Eps}^0$ by smooth morphisms and for which it is easy to find a germ which prorepresents it. 
Then the Main Theorem follows from the formal property of the relation of have the same type of singularities and from the closure of the set of germs with quadratic algebraic singularities with respect to this relation.

\section{Singularity type} \label{type of singularities}
Let $An$ be the category of analytic algebras and let $\hat An$ be the category of complete analytic algebras. 
We recall that an \emph{analytic algebra} is a $\C$-algebra which can be written in the form $\C\{x_1\ldots x_n\}/I$ and a morphism of analytic algebras is a local homomorphism of $\C$-algebras. 
\begin{definition} 
A homomorphism of rings $\psi:R\rightarrow S$ is called \emph{formally smooth} if, for every exact sequence of local artinian $R$-algebras: 
$ 0\rightarrow I\rightarrow B\rightarrow A\rightarrow 0$, such that $I$ is annihilated by the maximal ideal of $B$, the induced map 
$\Hom_{R}(S,B)\rightarrow \Hom_{R}(S,A)$ is surjective.
\end{definition}

\noindent We recall some facts about formally smooth morphisms of analytic algebras, which we use in this section. We start with the following equivalence of conditions (see \cite{Sernesi libro}, Proposition C.50):

\begin{prop} \label{smooth sernesi}
Let $\psi: R\rightarrow S $ be a local homomorphism of local noetherian $\C$-algebras conteining a field isomorphic to their residue field $\C$. Then the following conditions are equivalent:
\begin{itemize}
\item[-] $\psi$ is formally smooth,
\item[-] $\hat S$ is isomorphic to a formal power series ring over $\hat R$,
\item[-] the homomorphism $\hat \psi:\hat R\rightarrow \hat S$ induced by $\psi$ is formally smooth.
\end{itemize}
\end{prop}

\noindent Furthermore, we recall two Artin's important results (see \cite{A2}, Theorem 1.5a and Corollary 1.6):

\begin{theorem} \label{Artin's theorem}
Let $R$ and $S$ be analytic algebras and let $\hat R$ and $\hat S$ be their completions.  
Let $\bar\psi:R\rightarrow \hat S$ be a morphism of analytic algebras, then, for all $n\in \mathbb N$, there exists a morphism of analytic algebras $\psi_n:R\rightarrow S$, such that the following diagram is commutative:  
$$
\xymatrix{ R \ar[r]^-{\psi_n} \ar@{=}[d]^{\Id} & S \ar[r]^-{\pi_n} & S/ \mathfrak{m}^n_S \ar@{=}[d]^{\Id}\\
           R \ar[r]^-{\bar\psi} & \hat S \ar[r]^-{\pi_n}& S/ \mathfrak{m}^n_S  .  } $$
\end{theorem}
\begin{cor} \label{Artin's cor}
With the notation of Theorem \ref{Artin's theorem}, if in addition $\bar \psi$ induces an isomorphims $\hat{\bar\psi}: \hat R\rightarrow \hat S$, then $\psi_n$ is an isomorphism, provided $n\geq 2$.  
\end{cor}
\noindent Using these results, we can prove the following 

\begin{prop} \label{prop alla artin}
Let $R$ and $S$ be analytic algebras and let $\hat R$ and $\hat S$ be their completions. Let $\hat \psi:\hat R\rightarrow \hat S$ be a smooth morphism, then there exists a smooth morphism $R\rightarrow S$.  
\end{prop}
\begin{proof}
By Thereom \ref{smooth sernesi}, there exists an isomorphism $\hat \phi: \hat R[[x]] \rightarrow \hat S$, Corollary \ref{Artin's cor} implies that there exists an isomorphism $\phi: R\{x\}\rightarrow S$, which is obviously smooth by Theorem \ref{smooth sernesi}. Thus the morphism $\phi \circ i: R\hookrightarrow R\{x\} \to S$ is smooth. 
\end{proof}

\noindent To complete our study of analytic algebras, we prove the following

\begin{prop}
Let $R$ and $S$ be analytic algebras, such that 
\begin{itemize}
\item[-] $\dim_{\C} \mathfrak{m}_R/ \mathfrak{m}^2_R =\dim_{\C} \mathfrak{m}_S/ \mathfrak{m}^2_S$ and
\item[-] $R\{z_1,\ldots, z_N\} \cong S\{z_1, \ldots ,z_M\}$, for some $N$ and $M$,
\end{itemize}
then $R$ and $S$ are isomorphic.
\end{prop}
\begin{proof}
The first hypothesis implies that, in the isomorphism $R\{z_1,\ldots ,z_N\} \cong S\{z_1,\ldots ,z_M\}$, $N=M$. Moreover, proving the proposition by induction on $N$, the first hypothesis makes the inductive step trivial. Thus it is sufficient to prove the proposition for $N=1$.

Let $R= \C\{x_1,\ldots ,x_n\}/I $ and $S=\C \{y_1,\ldots ,y_m\}/J$ be analytic algebras, with $I \subset (x_1, \ldots ,x_n)^2$ and $J\subset (y_1,\ldots ,y_m)^2$.
Let $\phi:  \C\{ \underline{x} \}\{z\}/I \to  \C\{\underline{y} \}\{z\}/J$ be an isomorphism and let $\psi$ its inverse. Let $\phi(z)=\alpha z + \beta (\underline{y}) + \gamma(\underline{y}, z)$ and let $\psi(z)=az+ b(\underline{x})+c(\underline{x},z)$, where $\alpha, a \in \C$ are costants, $\beta$, $\gamma$, $b$ and $c$ are polynomial, $\gamma$ and $c$ do not contain degree one terms and, with a linear change of variables, we can suppose that $\phi$ and $\psi$ do not contain constant term.\\
If at least one between $\alpha$ and $a$ is different from zero, then the thesis follows easly. For example, if $\alpha \neq 0$, the image $\phi(z)$ satisfies the hypothesis of Weierstrass Preparation Theorem and so it can be written as $\phi(z)= (z+h(\underline{y}))\cdot u(\underline{y},z)$, where $u$ is a unit and $h(\underline{y})$ is a polynomial. Then $\phi$ is well defined and induces an isomorphism on quotients: $\phi: \C\{\underline{x}\}/I \to \C\{\underline{y}\}\{z\}/J\cdot (z+h(\underline{y}))\cong\C\{\underline{y}\}/J $.\\
Let's now analyse the case $\alpha=a=0$. Let $\nu: \C\{\underline{x}\}\{z\}/I \to \C\{\underline{x}\}\{z\}/I$ be a homomorphism defined by $\nu(x_i)=x_i$, for all $i$, and $\nu(z)= z+ b(\underline{x})$. It is obviously an isomorphims and the composition $\phi\circ\nu$ is an isomorphims from $\C\{\underline{x}\}\{z\}/I$ to $\C\{\underline{y}\}\{z\}/J$, such that $\phi \circ \nu(z)$ contains a linear term in $z$, thus, passing to the quotient, it induces an isomorphism $\C\{\underline{x}\}/I \cong \C\{\underline{y}\}/J$. 
\end{proof}

Now we consider the following relation between analytic algebras:
$$ R \propto S \mbox{ \ \ \ iff \ \ \ } \exists \ \ R\longrightarrow S \mbox{ \ \ \ formally smooth morphism,}$$
let $\sim$ be the equivalence relation between analytic algebras generated by $\propto$. 
We define an other equivalence relation:
$$ R \approx S \mbox{ \ \ \ iff \ \ \ }  R\{x_1 \ldots x_n \} \cong S\{y_1\ldots y_m \} \mbox{ \ \ \ are isomorphic, for some $n$ and $m$.}$$
The relation $\approx$ is the same as the relation $\sim$. 
Infact, if $R\sim S$, there exists a chain of formally smooth morphisms $R\rightarrow T_1 \leftarrow T_2 \rightarrow \ldots \rightarrow T_n \leftarrow S$, that, by Theorem \ref{smooth sernesi} and Corollary \ref{Artin's cor}, gives an isomorphism $R\{\underline{x}\}\cong S\{\underline{y}\}$, then $R\approx S$. Viceversa, if $R \approx S$, there exists an isomorphism $R\{\underline{x}\} \cong S\{\underline{y}\}$ which is formally smooth, thus we have the chain of formally smooth morphisms $R\rightarrow R\{\underline{x}\} \rightarrow S\{\underline{x}\} \leftarrow S$ and $R\sim S$. 

We consider the following relation between complete analytic algebras:
$$ \hat R \propto \hat S \mbox{ \ \ \ iff \ \ \ } \exists \ \ \hat R \rightarrow  \hat S \mbox{ \ \ \ formally smooth morphism,}$$
let $\sim$ be the equivalence relation between analytic algebras generated by $\propto$.
We define an other equivalence relation:
$$ \hat R \equiv \hat S \mbox{ \ \ \ iff \ \ \ }  \hat R[[x_1 \ldots x_n ]] \cong \hat S[[x_1\ldots x_m]] \mbox{ \ \ \ are isomorphic, for some $n$ and $m$.}$$
As before, the relation $\equiv$ is the same as the relation $\sim$. 
Furthermore, the equivalence relation $\sim$ on completions of analytic algebras coincides with the relation $\sim$ between the analytic algebras themselves, because obviously the two relations $\equiv$ and $\approx$ are the same.

\ \\
The opposite category of the category of analytic algebras $An^o$ is called the category of \emph{germs of analytic spaces}. The geometrical meaning of this definition is that a germ $A^o$ can be represented by $(X,x,\alpha)$, where $X$ is a complex space with a distinguished point $x$ and $\alpha$ is a fixed isomorphism of $\C$-algebras $\Oh_{X,x}\cong A$. Two triples, $(X,x,\alpha)$ and $(Y,y,\beta)$, are equivalent if there exists an isomorphism from a neighborhood of $x$ in $X$ to a neighborhood of $y$ in $Y$ which sends $x$ in $y$ and which induces an isomorphism $\Oh_{X,x}\cong \Oh_{Y,y}$.

Let $(X,x)$ and $(Y,y)$ be germs of analytic spaces, given by the analytic algebras $S$ and $R$ respectively, let $\Psi:(X,x)\rightarrow (Y,y)$ be a morphism of germs of analytic spaces and let $\psi:R\rightarrow S$ be the corresponding morphism of analytic algebras. 
\begin{definition}
The morphism $\Psi:(X,x)\rightarrow (Y,y)$ is called \emph{smooth} if the morphism $\psi:R\rightarrow S$ is formally smooth.
\end{definition}

We consider the following relation between germs of analytic spaces:\\
$$(X,x) \propto (Y,y) \mbox{ \ \ \ iff \ \ \ } \exists \ \ (X,x)\longrightarrow (Y,y) \mbox{ \ \ \ smooth morphism}$$
and we define $\sim$ to be the equivalence relation between germs of analytic spaces generated by the relation $\propto$. 
It is obvious that the relation $\sim$ defined between germs of analytic spaces is the same as the relation $\sim$ defined between their corresponding analytic algebras.
As in \cite{Vakil}, we give the following 
\begin{definition} \label{type sing sp ana}
The analytic spaces $(X,x)$ and $(Y,y)$ are said to have the same \emph{type of singularities} if they are equivalent under the relation $\sim$.  
\end{definition}

\noindent Our aim is to prove that the property that two germs of analytic spaces have the same type of singularities is formal, that is that it can be controlled at the level of functors.

\ \\
In all this paper we consider covariant functors
$\F:Art_{\C}\rightarrow Set$ from the category of local artinian $\C$-algebras with residue field $\C$ to the category of sets,
such that $\F(\C)=$ one point set. The functors of this type are called \emph{functors of Artin rings}. In the following we recall some basic notions about these functors. 

Let $\F$ be a functor of Artin rings. The \emph{tangent space} to $\F$ is the set $\F(\C[\epsilon])$, where $\C[\epsilon]$ is the local artinian $\C$-algebra of dual numbers, i.e. $\C[\epsilon]=\C[x]/(x^2)$ . It can be proved that $\F(\C[\epsilon])$ has a structure of $\C$-vector space (see \cite{Schlessinger}, Lemma 2.10).

An \emph{ostruction theory} $(V,v_e)$ for $\F$ is the data of a $\C$-vector space $V$, called \emph{obstruction space}, and, for every exact sequence in the category $Art_{\C}$:
$$e: 0\longrightarrow I  \longrightarrow  B \longrightarrow A \longrightarrow 0, $$ 
such that $I$ is annihilated by the maximal ideal of $B$,
a map $v_e: \F(A)\rightarrow V\otimes_{\C} I$ called \emph{obstruction map}. 
The data $(V,v_e)$ have to satisfy the following conditions:
\begin{itemize}
\item[-] if $\xi\in \F(A)$ can be lifted to $\F(B)$, then $v_e(\xi) = 0$,
\item[-] (base change) for every morphism $f:e_1\rightarrow e_2$ of small extensions, i.e. for every commutative
diagram
$$ 
\xymatrix{ e_1:  & 0 \ar[r] & I_1 \ar[d]^-{f_I} \ar[r] & B_1 \ar[d]^-{f_B}\ar[r] & A_1 \ar[d]^-{f_A}\ar[r] & 0 \\
           e_2:  & 0 \ar[r] & I_2  \ar[r] & B_2 \ar[r] & A_2 \ar[r] & 0  }   $$
then $v_{e_2}(f_A(\xi)) = (\Id_V \otimes f_I)(v_{e_1}(\xi))$, for every $\xi \in \F(A_1)$.
\end{itemize} 
An obstruction theory $(V,v_e)$ for $\F$ is called \emph{complete} if the converse of the first
item above holds, i.e. the lifting of $xi\in \F(A)$ to $\F(B)$ exists if and only if the obstruction $v_e(\xi)$ vanishes.

For morphisms of functors we have the following notion of smoothness:
\begin{definition}
Let $\nu:\F\rightarrow \G$ be a morphism of functors, it is said to be \emph{smooth} if, for every surjective homomorphism $B\rightarrow A$ in the category $Art_{\C}$, the induced map $\F(B)\rightarrow \G(B)\times_{\G(A)} \F(A)$ is surjective.
\end{definition}

Let $\F$ be a functor of Artin rings, a \emph{couple} for $\F$ is a pair $(A,\xi)$, where $A\in Art_{\C}$ and $\xi\in \F(A)$. A couple $(A,\xi)$ for $\F$ induces an obvious morphism of functors, $\Hom(A,-)\rightarrow \F$, which associates, to every $B\in Art_{\C}$ and $\phi\in \Hom(A,B)$, the element $\phi(\xi)\in \F(B)$.\\
We can extend the functor $\F$ to the category $\widehat{Art}_{\C}$ of local artinian complete $\C$-algebras with residue field $\C$ by the formula $\widehat\F(A)=\displaystyle \lim_{\leftarrow}\F(A/\mathfrak{m}^n)$. 
A \emph{procouple} for $\F$ is a pair $(A,\xi)$, where $A\in \widehat{Art}_{\C}$ and $\xi\in \widehat{\F}(A)$. It induces an obvious morphism of functors:  $\Hom(A,-)\rightarrow \F$.

\begin{definition}
A procouple $(A,\xi)$ for a functor $\F$ is called a \emph{prorepresentable hull} of $\F$, or just a \emph{hull} of $\F$, if the induced morphism $\Hom(A,-)\rightarrow \F$ is smooth and the induced map between tangent spaces $\Hom(A,\C[\epsilon])\rightarrow \F(\C[\epsilon])$ is bijective. \\
A functor $\F$ is called \emph{prorepresentable} by the procouple $(A,\xi)$ if the induced morphism $\Hom(A,-)\rightarrow \F$ is an isomorphism of functors.
\end{definition}

The existence, for a functor of Artin rings $\F$, of a procouple which is a hull of it or which prorepresents it is regulated by the well known 
\emph{Schlessinger conditions} (see \cite{Schlessinger}, Theorem 2.11). Moreover a functor which satisfies the two of the Schlessinger's conditions (H1) and (H2) is said to have a \emph{good deformation theory}. We do not precise this concept, because all functors that apper in the following are functors with a good deformation theory and all the ones involved in the proof of the Main Theorem have hull.   

With above notions, we can state the following Standard Smoothness Criterion (see \cite{Man Pisa}, Proposition 2.17):

\begin{theorem} \label{criterio liscezza}
Let $\nu:\F\rightarrow \G$ be a morphism of functors with a good deformation theory.  
Let $(V,v_e)$ and $(W,w_e)$ be two obstruction theories for $\F$ and $\G$ respectively. If:
\begin{itemize}
	\item[-] $(V,v_e)$ is a complete obstruction theory,
	\item[-] $\nu$ is injective between obstructions,
	\item[-] $\nu$ is surjective between tangent spaces, 
\end{itemize}
then $\nu$ is smooth.
\end{theorem}
\noindent For morphisms between $\Hom$ functors the following proposition holds (see \cite{Schlessinger}, Proposition 2.5):
\begin{prop} \label{smooth schlessinger}
Let $\hat \psi:\hat R\rightarrow \hat S$ be a local homomorphism of local noetherian complete $\C$-algebras, let $\hat \phi: \Hom(\hat S,-)\rightarrow \Hom(\hat R,-)$ be the morphism of functors induced by $\hat\psi$. Then $\hat \phi$ is smooth if and only if $\hat S$ is isomorphic to a formal power series ring over $\hat R$.
\end{prop}

\ \\ 
We consider the following relation between functors:
$$\F \propto \G \mbox{ \ \ \ iff \ \ \ } \exists \ \ \F\longrightarrow \G \mbox{ \ \ \ smooth morphism}$$
and we define $\sim$ to be the equivalence relation generated by $\propto$. 
\begin{definition} \label{rel funtori}
The functors $\F$ and $\G$ are said to have the same \emph{type of singularity} if they are equivalent under the relation $\sim$.
\end{definition}
Now we want to link Definitions \ref{type sing sp ana} and \ref{rel funtori} in the case the functors considered have hulls. 
Let $\F$ and $\G$ be two functors with hulls, given by the germs of analytic spaces $(X,x)$ and $(Y,y)$, defined by the analytic algebras $S$ and $R$ respectively.
If $(X,x)\sim (Y,y)$, or equivalently, if $R\sim S$, there exists a chain of smooth morphisms of analytic algebras $S\leftarrow T_1\rightarrow T_2\leftarrow \ldots\rightarrow R$ that induces a chain of smooth morphisms of functors $\F\leftarrow \Hom(\hat S,-)\rightarrow \Hom(\hat T_1,-)\leftarrow \Hom(\hat T_2,-)\ldots \Hom(\hat R,-) \rightarrow \G$, by Propositions \ref{smooth sernesi} and \ref{smooth schlessinger} and by Definition of hull, thus $\F\sim \G$. 
For the other implication we need the following 
\begin{prop} \label{prop formalita}
Let $\F$ and $\G$ be two functors with hulls given by the germs of analytic spaces $(X,x)$ and $(Y,y)$ respectively and let $\phi:\F\rightarrow \G$ be a smooth morphism. Then there exists a smooth morphism between the two germs $(X,x)$ and $(Y,y)$.    
\end{prop}
\begin{proof}
Let $S$ and $R$ be the analytic algebras that define the germ $(X,x=0)$ and $(Y,y=0)$ respectively. 
By hypothesis, we have the following diagram: 
$$
\xymatrix{ \F \ar[r]^{\phi} & \G\\
           \Hom(\hat{S},-) \ar[u]^{\alpha} \ar@{-->}[r]^-{\hat\phi}& \Hom(\hat{R},-) \ar[u]^{\beta}   } $$
where, by definition of hull, $\alpha$ and $\beta$ are smooth morphism and they are bijective on tangent spaces. 
Then, by smoothness, there exists a morphism $\hat\phi: \Hom(\hat{S},-)\rightarrow \Hom(\hat{R},-)$ that makes the diagram commutative. 
By hypothesis on $\alpha$, $\beta$ and $\phi$, it is surjective on tangent spaces and it is injective on obstruction spaces. Thus $\hat\phi$ is a smooth morphism, by the Standard Smoothness Criterion (Theorem \ref{criterio liscezza}).\\
The morphism $\hat\phi$ determines uniquely an homomorphism $\hat\psi:\hat{R}\rightarrow \hat{S}$, which is formally smooth, by Propositions \ref{smooth sernesi} and \ref{smooth schlessinger}. 
Now, by Proposition \ref{prop alla artin}, there exists a formally smooth morphism $\psi:R\rightarrow S$ and so a smooth morphism between the germs $(X,x)$ and $(Y,y)$.
\end{proof}
\noindent Now, if $\F\sim \G$, there exists a chain of smooth morphisms of functors $\F \leftarrow \mathcal H_1 \rightarrow \mathcal H_2 \leftarrow \ldots \rightarrow \G$. Then $\H_i$ necessary have hulls, we indicate with $T_i$ the complete analytic algebra that is an hull for $\H_i$. By Proposition \ref{prop formalita}, the chain of smooth morphisms of functors gives a chain of smooth morphisms of complete analytic algebras $\hat S\rightarrow T_1 \leftarrow T_2\rightarrow \ldots \leftarrow \hat R$, thus $\hat S\equiv \hat R$, so, as we have observed, $S\sim R$ and $(X,x) \sim (Y,y)$.

\ \\
Now we return to germs of analytic spaces and we concentrate our interests on quadratic algebraic singularities.
\begin{definition} \label{def quadratic alg sing for spaces}
Let $X$ be a complex affine scheme, it is said to have \emph{quadratic algebraic singularities} if it is defined by finitely many quadratic homogeneous polynomials.\\
Let $\mathcal X$ be an analytic space, it is said to have \emph{quadratic algebraic singularities} if it is locally isomorphic to complex affine schemes with quadratic algebraic singularities. 
\end{definition}
\noindent For germs of analytic spaces we want to prove the following 
\begin{theorem} \label{teo sing quadr}
Let $(X,0)$ and $(Y,0)$ be two germs of analytic spaces and let $\phi:(X,0)\longrightarrow (Y,0)$ be a smooth morphism. Then $(X,0)$ has quadratic algebraic singularities if and only if $(Y,0)$ has quadratic algebraic singularities.
\end{theorem}
\noindent We need the following 
\begin{lemma} \label{lemma}
Let $(X,0)$ and $(Y,0)$ be two germs of analytic spaces and let $\phi:(X,0)\longrightarrow (Y,0)$ be a smooth morphism. Let $X\subset \C^N$ and let $H=\{x\in \C^N \mid h(x)=0 \}$ be an hypersurface of $\C^N$, such that:
\begin{itemize}
	\item[-] $dh(0)\neq 0$
	\item[-] $TH \not\supset  T\phi^{-1}(0)$,
\end{itemize}
then: $\phi|_{X\cap H}:(X\cap H,0)\longrightarrow (Y,0)$ is a smooth morphism.
\end{lemma}
\begin{proof}
Let $(X,0)$ and $(Y,0)$ be defined by $\C\{x_1,\ldots,x_n\}/I$ and $\C\{y_1,\ldots,y_m\}/J$ respectively. Since $\phi$ is smooth, $\Oh_{X,0}$ is a power series ring over $\Oh_{Y,0}$, i.e. $\Oh_{X,0}\cong \Oh_{Y,0}\{t_1,\ldots,t_s\}$, for some $s$.\\
Let $X'=X\cap H$ be the intersection, then $\Oh_{X',0}\cong \Oh_{X,0}/(h)$. If $g$ corresponds to $h$ by the isomorphism $\Oh_{X,0}\cong \Oh_{Y,0}\{t_1,\ldots,t_s\}$, then $\Oh_{X',0}\cong \Oh_{Y,0}\{t_1,\ldots,t_s\}/(g)$.\\
The hypothesis $dh(0)\neq 0$ becomes $dg(0)\neq 0$, that implies that there exists an indeterminate between $y_i$ and $t_i$, such that the partial derivative of $g$ with respect to this indeterminate calculated in zero is not zero. Moreover, the hypothesis $TH\not\supset T\phi^{-1}(0)$ implies that this indeterminate must be one of the $t_i$, for example $t_{\overline{i}}$. \\
Thus, using the Implicit Function Theorem, we obtain $ \Oh_{X',0}\cong \Oh_{Y,0}\{t_1,\ldots,t_s\}/(g)\cong \Oh_{Y,0}\{t_1,\ldots,\widehat{t_{\overline i}},\ldots,t_s \}$ and $\phi|_{X'}$ is a smooth morphism.
\end{proof}
\noindent Now we can prove Theorem \ref{teo sing quadr}:
\begin{proof}
We start by assuming that $(Y,0)$ has quadratic algebraic singularities, so $(Y,0)$ is defined by the analytic algebra $\C\{y_1,\ldots, y_m\}/J$, where $J$ is an ideal generated by quadratic polynomials.
Since $\phi$ is smooth, we have $\Oh_{X,0}\cong \Oh_{Y,0}\{t_1,\ldots, t_s\} \cong \C\{y_1,\ldots, y_m\}\{t_1,\ldots,t_s\}/J$, for some $s$, and $X$ has quadratic algebraic singularities.

Now we prove the other implication. 
Let $\C\{x_1,\ldots,x_n\}/I$ and $\C\{y_1,\ldots,y_m\}/J$ be the analytic algebras, that define the germs $(X,0)$ and $(Y, 0)$ respectively, where $I$ is an ideal generated by quadratic polynomials.
We can assume that $\phi$ is not an isomorphism, otherwise the theorem is trivial.
Since $\phi$ is smooth, $\Oh_{X,0}\cong \Oh_{Y,0}\{t_1,\ldots,t_s\}$, for some $s>0$. \\
Now we can intersect $X\subset \C^N_x$ with hyperplanes $h_1, \ldots, h_s$ of $\C^N_x$, which correspond, by the isomorphism $\Oh_{X,0}\cong \Oh_{Y,0}\{t_1,\ldots,t_s\}$, to the hyperplanes of equations $t_1=0, \ldots, t_s=0$ of $\C^{m+s}_{y,t}$ and we call the intersection $X'$. 
Then $(X',0)$ has quadratic algebraic singularities. Moreover, by lemma \ref{lemma}, $\phi$ restricted to $(X',0)$ is a smooth morphism and it is bijective because $\Oh_{X',0}\cong \Oh_{X,0}/(h_1,\ldots,h_s)\cong \Oh_{Y,0}\{t_1,\ldots,t_s\}/(t_1,\ldots,t_s)\cong \Oh_{Y,0}$.
Thus $(Y,0)$ has quadratic algebraic singularities. 
\end{proof} 
\noindent This theorem assures that the set of germs of analytic spaces with quadratic algebraic singularities is closed under the relation $\sim$ and so it is a union of equivalent classes under this relation. Moreover we know that the relation $\sim$ defined between functors with hulls is the same as the relation $\sim$ defined between their germs.
Thus is natural to introduce the following definition for functors:
\begin{definition}
Let $\F$ be a functor with hull the germ of analytic space $(X,x)$. It is said to have \emph{quadratic algebraic singularities} if $(X,x)$ has quadratic algebraic singularities.
\end{definition}
\noindent This definition is indipendent by the choice of the germ of analytic space which is a hull of $\F$, because the isomorphisms class of a hull is uniquely determined.

\section{Brill-Noether stratification of $\M$} \label{def di N}
Let $X$ be a compact complex K\"{a}hler manifold of dimension $n$ and let $\M$ be the moduli space of locally free sheaves of $\Oh_X$-modules on $X$ which are stable and flat.
In the first part of this section we concentrate our interests to the study of subspaces of the moduli space $\M$, defined globally as sets in the following way:
\begin{definition}\label{N as set}
Let $h_i\in \mathbb{N}$ be fixed integers, for all $i= 0\ldots n$, we define:
$$ \mathcal N(h_0\ldots h_n)= \{\Eps' \in \M \mid \dim H^i(X,\Eps')=h_i  \}. $$
\end{definition}
It is obvious that, for a generic choice of the integers $h_i\in \mathbb N$, the subspace $\mathcal N(h_0\ldots h_n)$ is empty, from now on we fix our attention on non empty ones.
 
Let $\Eps\in\M$ be one fixed stable and flat locally free sheaf of $\Oh_X$-modules on $X$ and let $h_i=\dim H^i(X,\Eps)$. Let $\U\to M\times X$ be the universal Kuranishi family of deformations of $\Eps$, parametrized by the germ of analytic space $M$, which is isomorphic to a neighbourhood of $\Eps$ in $\M$ (see \cite{Nadel} for the construction). Let $\nu:M\times X\rightarrow M$ be projection, thus, for all $\Eps'\in M$, we have $\nu^{-1}(\Eps')\cong X$ and $\U|_{\nu^{-1}(\Eps')}=\U|_{\Eps'}\cong \Eps'$.

Now let's define the germ of the strata $\mathcal N(h_0\ldots h_n)$ at its point $\Eps$.
Since, for all $i=0\ldots n$, the function $\Eps'\in \M \rightarrow \dim H^i(X,\Eps')\in \N$ is upper semicontinuos, for Semicontinuity Theorem (see \cite{Hartshorne}, Theorem 12.8, ch.III), the set $U_i=\{\Eps' \in \M \mid \dim H^i(X,\Eps')\leq h_i \} $ and the intersection $U=\bigcap_{i=0\ldots n} U_i= \{\Eps' \in \M \mid \dim H^i(X,\Eps')\leq h_i,\textrm{ for } i=0\ldots n  \}$ are open subsets of $\M$.

For all $i=0\ldots n$, let $N_i(\Eps)=V(F_{h_i-1}(R^i\nu_*\U))=\{\Eps'\in M \mid \dim R^i\nu_*\U \otimes_{\Oh_{M}} k(\Eps') > h_i-1 \}$ be the closed subschemes of $M$ defined by the sheaf of ideals $F_{h_i-1}(R^i\nu_*\U)$, which is the sheaf of $(h_i-1)$-th Fitting ideals of the sheaf of $\Oh_M$-modules $R^i\nu_*\U$. Let $N(\Eps)=\bigcap_{i=0\ldots n}N_i(\Eps)$ be the closed subscheme of $M$ given by the intersection of the previous ones. 

\begin{definition} \label{N as scheme}
The germ of the strata $\mathcal N(h_0\ldots h_n)$ at its point $\Eps$ is given by:
$$ U\cap N(\Eps) = \{ \Eps' \in \M \mid \dim H^i(X,\Eps')\leq h_i, \ \ \forall i=0\ldots n\} \cap \bigcap_{i=0\ldots n} V(F_{h_i-1}(R^i\nu_*\mathcal U)).$$
\end{definition}

\begin{remark}
We observe that the support of the germ of the strata $\mathcal N(h_0\ldots h_n)$ at $\Eps$, defined in \ref{N as scheme}, coincide with a neighbourhood of $\Eps$ in set given in definition \ref{N as set}. 
Infact, for the Theorem of Cohomology and Base Change (see \cite{Hartshorne}, Theorem 12.11, ch.III), we have $R^n\nu_*\U \otimes_{\Oh_{M}} k(\Eps') \cong H^n(X,\Eps')$, then the condition which defines $N_n(\Eps)$ becomes $\dim H^n(X,\Eps') \geq h_n$ and the ones which define the intersection $U\cap N_n(\Eps)$ become $ \dim H^n(X,\Eps') = h_n$ and $\dim H^i(X,\Eps')\leq h_i$, for all $i=0\ldots n-1$. Applying iterative the Theorem of Cohomology and Base Change, we obtain $U\cap N(\Eps)=\{\Eps' \in M \mid \dim H^i(X,\Eps')=h_i, \textrm{ for } i=0\ldots n \} $ as we want.
\end{remark}

Now we prove the following:
\begin{prop} 
The germ of the strata $\mathcal N(h_0\ldots h_n)$ at $\Eps$ is the base space of a Kuranishi family of deformations of $\Eps$ which preserve the dimensions of cohomology spaces.
\end{prop}
\begin{proof} 
Let $\mathcal F$ be a locally free sheaf of $\Oh_{T\times X}$-module on $T\times X$ which is a deformation of the sheaf $\Eps$ over the analytic space $T$ that preserve the dimensions of cohomology spaces. 
If the morphism $g: T\rightarrow M$ such that $(g\times \Id_X)^*\U\cong \mathcal F$, whose existence is assured by the universality of $\U$, can be factorized as in the following diagram:
$$
\xymatrix{ \mathcal F \cong (g\times \Id_X)^* \U  \ar[r] \ar[d] & \U \ar[d] \\
           T\times X \ar[r]^-{g\times \Id_X} \ar@{-->}[dr]_-{h\times \Id_X} & M \times X\\
                                                                        & (\mathcal N(h_0\ldots h_n)\cap M) \times X, \ar[u]_-{i\times \Id_X}} $$
then $\mathcal F\cong (g\times \Id_X)^* \U\cong (h\times \Id_X)^*(i\times \Id_X)^* \U \cong (h\times \Id_X)^* \U|_{(\mathcal N(h_0\ldots h_n) \cap M) \times X}$ and the restriction of $\U$ to $(\mathcal N(h_0\ldots h_n)\cap M )\times X$ satisfies the universal property.

Let's analise the pullback via the map $g$ of the sheaf of ideals $F_k(R^i\nu_*\U)$, for $i=0\ldots n$, which defines locally $\mathcal N(h_0\ldots h_n)\cap M$. Since Fitting ideals commute with base change, for all $i$ and $k$ we have $ g^* (F_k(R^i\nu_* \U ))= F_k(g^* R^i\nu_* \U)$. Let's consider the diagram:
$$
\xymatrix{ \mathcal F\cong (g\times \Id_X)^*\U \ar[r] \ar[d] & \U \ar[d] \\
T\times X  \ar[r]^-{g\times \Id_X}\ar[d]^-{\mu} & \imm g\times X \subset M \times X \ar[d]^-{\nu} \\      
T \ar[r]^-{g}  &  \imm g\subset M ;  } $$
using the Theorem of Cohomology and Base Change,
since for all $i=0\ldots n$ the functions $\Eps'\in \imm g\subset M \rightarrow h^i(X,\Eps')\in \N$ are costant, we have $F_k(g^* R^i\nu_* \U)\cong F_k(R^i\mu_* (g\times \Id_X)^*\U) \cong F_k(R^i\mu_*\mathcal F)$, and since $\mathcal F$ is a deformation which preserves the dimensions of cohomology spaces, the sheaves $R^i\mu_* \mathcal F$ are locally free and so the Fitting ideals $F_k(R^i\mu_*\mathcal F)$ are equal to zero. Then $g^*F_k(R^i\nu_*\U)$ is equal to zero, as we want.
\end{proof}

In the following part of this section we study infinitesimal deformations of $\Eps$ that preserve the dimensions of its cohomology spaces, explaining this condition and defining precisely the deformation functor associated to this problem.

\begin{definition}
A \emph{deformation} of the sheaf $\Eps$ on the manifold $X$ parametrized by an analytic space $S$ with a fixed point $s_0$ is the data of a locally free sheaf $\Eps'$ of $(\Oh_{X \times S})$-modules on $X\times S$ and a morphism $\Eps'\rightarrow \Eps$ inducing an isomorphism between $\Eps'|_{X\times s_0}$ and $\Eps$. 
\end{definition}
Let $\pi:X\times S\rightarrow S$ be the projection, then, for all $s\in S$, $\Eps'|_{\pi^{-1}(s)}=\Eps'|_{X\times s}=\Eps'_s$ is a locally free sheaf on $X$ and so it makes sense to calculate the cohomology spaces of these sheaves, $H^i(\Eps'_s)$. 

By the Theorem of Cohomology and Base Change, the condition that, for all $i\in \N$, $\dim H^i(\Eps'_s)$ is costant when $s$ varies in $S$, is equivalent to the condition that, for all $i\in \N$, the direct image $R^i\pi_*\Eps'$ is a locally free sheaf on $S$, and in this case we have that the fibre $R^i\pi_*\Eps'\otimes k(s)$ is  isomorphic to $H^i(\Eps'_{s})$.
\begin{definition} 
An \emph{infinitesimal deformation} of the sheaf $\Eps$ on the manifold $X$ over a local artinian $\C$-algebra $A$ with residue field $\C$ is the data of a locally free sheaf $\Eps_A$ of $(\Oh_X \otimes A)$-modules on $X\times \Spec A$ and a morphism $\Eps_A\rightarrow \Eps$ inducing an isomorphism $\Eps_A \otimes_A \C\cong \Eps$. 
\end{definition}
In the case of an infinitesimal deformation $\Eps_A$, we can replace the condition that the dimensions of the cohomology spaces are costant along the fibres of the projection $\pi:X\times \Spec A\rightarrow \Spec A$, with the condition that the direct images $R^i\pi_*\Eps_A$ are locally free sheaves, and in this case we have isomorphisms  $R^i\pi_*\Eps_A\otimes \C\cong H^i(\Eps)$.
We observe also that $H^i(\Eps_A)\cong R^i\pi_*\Eps_A(\Spec A)$. 

Our aim is to study this type of infinitesimal deformations of the sheaf $\Eps$ on $X$ and the functor associated to them, defined in the following
\begin{definition} \label{prima def di F}
Let $\Def^0_{\Eps}: Art_{\C}\rightarrow Set$ be the covariant functor defined, 
for all $A\in Art_{\C}$, by:
$$
\Def^0_{\Eps}(A)= \left\{\Eps_A \left| \begin{array}{l} \Eps_A \textrm{ is a deformation of the sheaf } \Eps \textrm{ over } A \\
R^i\pi_*\Eps_A \textrm{ is a locally free sheaf on } \Spec A \textrm{ for all } i\in \N   
\end{array}  \right\}\right. / \sim .
$$
\end{definition}
We note that, since $\mathcal N(h_0\ldots h_n)$ is a moduli space, the functor of deformations of the sheaf $\Eps$ which preserve the dimensions of cohomology spaces is prorepresented by $\mathcal N(h_0\ldots h_n)$.   

To give now two equivalent interpretation of $\Def^0_{\Eps}$, which are useful in the following, we start by defining 
a \emph{deformation} of a $\C$-vector space $V$ over a local artinian $\C$-algebra $A$ which is the data of a flat $A$-module $V_A$,  such that the projection onto the residue field induces an isomorphism $V_A\otimes_A \C\cong V$.
It is easy to see that every deformation of a vector space $V$ over $A$ is trivial, i.e. it is isomorphic to $V\otimes A$. 

If $\Eps_A$ is an infinitesimal deformation of $\Eps$ over $A$, it belongs to $\Def^0_{\Eps}(A)$, as defined in (\ref{prima def di F}), if and only if it is such that  $H^i(\Eps_A)$ are free $A$-modules, that is the same as flat $A$-modules since $A$ is local artinian, and $H^i(\Eps_A)\otimes_A \C\cong H^i(\Eps)$. Thus we have the two following equivalent definitions for the functor $\Def^0_{\Eps}$:
\begin{definition}
The functor $\Def^0_{\Eps}$ is defined, for all $\A\in Art_{\C}$, by:
$$ 
\Def^0_{\Eps}(A)= \left\{\Eps_A \left| \begin{array}{l} \Eps_A \textrm{ is a deformation of the sheaf } \Eps \textrm{ over } A \\
H^i(\Eps_A) \textrm{ is a deformation of } H^i(\Eps) \textrm{ over } A \textrm{ for all } i\in \N  
\end{array} \right\}\right. / \sim .
$$
or equivalently by:
\begin{equation} \label{def funtore F}
\Def^0_{\Eps}(A)= \left\{\Eps_A \left| \begin{array}{l} \Eps_A \textrm{ is a deformation of the sheaf } \Eps \textrm{ over } A \\
H^i(\Eps_A) \textrm{ is isomorphic to } H^i(\Eps)\otimes A  \textrm{ for all } i\in \N
\end{array}\right\}\right. / \sim .
\end{equation}
\end{definition}

\section{Definition of $\Def^0_{\Eps}$ using DGLAs} \label{F con DGLA}
We are interested in the study of the functor $\Def_{\Eps}^0$ and in this section we link it with the theory of deformations via DGLAs.\\

We start with some reminds about this theory. Let $L$ be a differential graded Lie algebra (DGLA), then it is defined a deformation functor $\Def_L: Art_{\C}\to Set$ canonically associated to it (see \cite{Man Pisa}, section 3). 
\begin{definition}
For all $(A,\mathfrak{m}_A)\in Art_{\C}$, we define:
$$\Def_L(A)=\frac{\MC_L(A)}{\sim_{\textrm{gauge}}},$$
where:
$$\MC_L(A)=\left\{x\in L^1\otimes \mathfrak{m}_A \mid dx+\frac{1}{2}[x,x]=0\right\}$$
and the gauge action is the action of $\exp (L^0\otimes \mathfrak{m}_A)$ on $\MC_L(A)$, given by:
$$e^a * x= x+\sum_{n=0}^{+\infty} \frac{([a,-])^n}{(n+1)!}([a,x]-da).$$
\end{definition} 
We recall that the tangent and an obstruction space to the deformation functor $\Def_L$ are the first and the second cohomology spaces of the DGLA $L$, $H^1(L)$ and $H^2(L)$. 
Moreover the deformation functor $\Def_L$ is a functor with good deformation theory and, if $H^1(L)$ is finite dimensional, it has a hull, but in general it is not prorepresentable.\\
If the functor of deformations of a geometrical object $\mathcal{X}$ is isomorphic to the deformation functor associated to a DGLA $L$, then we say that $L$ governs the deformations of $\mathcal{X}$.

Let $\chi:L\rightarrow M$ be a morphism of DGLAs which respects the DGLA structures of $L$ and $M$, then it is defined a deformation functor $\Def_{\chi}:Art_{\C}\to Set$ canonically associated to it (see \cite{Man sottovar}, section 2).
\begin{definition}
For all $(A, \mathfrak{m}_A)\in Art_{\C}$, we define:
$$\Def_\chi(A)=\frac{\MC_\chi(A)}{\sim_{\textrm{gauge}}},$$
where:
$$MC_{\chi}(A) =\left\{(x,e^a)\in (L^1\otimes \mathfrak{m}_A) \times \exp(M^0\otimes \mathfrak{m}_A ) \mid dx+\frac{1}{2}[x,x]=0, e^a* \chi(x)=0\right\}$$
and the gauge action is the action of $\exp(L^0\otimes \mathfrak{m}_A)\times \exp(dM^{-1}\otimes \mathfrak{m}_A)$ on $\MC_{\chi}(A)$, given by:
$$(e^l,e^{dm})*(x,e^a)= (e^l*x, e^{dm}e^a e^{-\chi(l)}).$$ 
\end{definition}
To write the tangent space and an obstruction space of a deformation functor associated to a morphism of DGLAs, we start by recalling that,  
if $\chi:L\rightarrow M$ is a morphism of DGLAs, the \emph{suspension of the mapping cone} of $\chi$ is defined to be the complex $C_{\chi}^i=L^{i}\oplus M^{i-1}$ with differential $\delta(l,m)=(d_Ll, \chi(l)-d_Mm)$.

Then the tangent space to the functor $\Def_\chi$ is the first cohomology space of the suspension of the mapping cone of the morphism $\chi$, $H^1(C_\chi)$, an obstruction space of $\Def_{\chi}$ is the second cohomology space, $H^2(C_\chi)$, and the obstruction theory for $\Def_{\chi}$ is complete.  \\
Moreover $\Def_{\chi}$ is a functor with good deformation theory and, if $H^1(C_{\chi})$ is finite dimensional, it has a hull, but in general it is not prorepresentable.\\
We also observe that
every commutative diagram of differential graded Lie algebras 
\begin{equation}\label{d} \xymatrix{ L \ar[r]^-{f} \ar[d]^-{\chi} & H \ar[d]^-{\eta} \\
              M \ar[r]^{f'} & I } \end{equation}
induces a morphism between the cones $C_\chi\rightarrow C_\eta$ and a morphism of functors $\Def_{\chi}\rightarrow \Def_{\eta}$, for which the following \emph{Inverse Function Theorem} (see \cite{Man sottovar}, Theorem 2.1) holds:  
\begin{theorem}  \label{inverse function theorem}
If the diagram (\ref{d}) induces a quasi isomorphism between the cones $C_\eta\rightarrow C_\chi$, then the induced morphism of functor $\Def_{\chi}\rightarrow \Def_{\eta}$ is an isomorphim. 
\end{theorem}             
\ \\
Now we return to our situation, so let $X$ be a compact complex K\"{a}hler manifold of dimension $n$ and let $\Eps$ be a stable and flat locally free sheaf of $\Oh_X$-modules on $X$, with $\dim H^i(X, \Eps)=h_i$, for all $i=0\ldots n$.

Let $A_X^{(0,*)}(\End\Eps)$ be the DGLA of the $(0,*)$-forms on $X$ with values in the sheaf of the endomorphisms of $\Eps$. It can be proved the following result (see \cite{Fuk}, Theorem 1.1.1):
\begin{prop}
The deformation functor $\Def_{A_X^{(0,*)}(\End\Eps)}$ is isomorphic to the functor of deformations of the sheaf $\Eps$, $\Def_{\Eps}$. The isomorphism is given, for all $A\in Art_\C$, by 
$$\begin{array}{rll}
\Def_{A_X^{(0,*)}(\End\Eps)}(A)& \longrightarrow &\Def_{\Eps}(A) \\
x& \longrightarrow & \ker(\deltabar+x)
\end{array}$$
\end{prop}

Let $(A_X^{(0,*)}(\Eps), \deltabar)$ be the complex of the $(0,*)$-forms on $X$ with values in the sheaf $\Eps$ with the Dolbeault differential and let $\Hom^*(A_X^{(0,*)}(\Eps),A_X^{(0,*)}(\Eps))$ be the DGLA of the homomorphisms of this complex. 

We recall that a \emph{deformation} of a complex of vector spaces $(V^i,d)$ over a local artinian $\C$-algebra $A$ with residue field $\C$ is a complex of $A$-modules of the form $(V^i\otimes A,d_A)$, such that the projection onto the residue field induces an isomorphism between $(V^i\otimes A,d_A)$ and $(V^i,d)$.

It is easy to prove the following result (see \cite{Man Seattle}, pages 3-4):
\begin{prop} 
The deformation functor $\Def_{\Hom^*(A_X^{(0,*)}(\Eps),A_X^{(0,*)}(\Eps))}$ is isomorphic to the functor of deformations of the complex $(A_X^{(0,*)}(\Eps), \deltabar)$, $\Def_{(A_X^{(0,*)}(\Eps), \deltabar)}$. The isomorphism is given, for all $A\in Art_\C$, by:
$$ \begin{array}{rll}  \Def_{\Hom^*(A_X^{(0,*)}(\Eps),A_X^{(0,*)}(\Eps))}(A)& \longrightarrow & \Def_{(A_X^{(0,*)}(\Eps), \deltabar)}(A)\\
x &\longrightarrow & (A_X^{(0,*)}(\Eps)\otimes A, \deltabar+x).
\end{array}
$$
\end{prop}

Let $\chi: A_X^{(0,*)}(\End\Eps)\rightarrow \Hom^*(A_X^{(0,*)}(\Eps),A_X^{(0,*)}(\Eps))$ be the natural inclusion of DGLAs and let $\Def_{\chi}$ be the deformation functor associated to $\chi$. 
Let $(x,e^a)\in \MC_{\chi}(A)$, for $A\in Art_{\C}$. 
Since $x\in A_X^{(0,1)}(\End\Eps)\otimes \mathfrak{m}_A$ satisfies the Maurer-Cartan equation, it gives a deformation $\Eps_A=\ker(\deltabar+x)$ of $\Eps$ over $A$. While $e^a\in \exp(\Hom^0(A_X^{(0,*)}(\Eps),A_X^{(0,*)}(\Eps))\otimes \mathfrak{m}_A)$ gives a gauge equivalence between $\chi(x)=x$ and zero in the DGLA $\Hom^*(A_X^{(0,*)}(\Eps),A_X^{(0,*)}(\Eps))$. 
Thus $e^a$ is an isomorphism between the two correspondent deformations of the complex $(A_X^{(0,*)}(\Eps),\deltabar)$ 
or equivalently $e^a$ is an isomorphism between the cohomology spaces $H^i(\Eps_A)$ and  $H^i(\Eps)\otimes A$, for all $i\in \N$. 
Thus:
$$
\Def_{\chi}(A)= \left\{(\Eps_A,f^i_A) \left| \begin{array}{ll} \Eps_A & \textrm{is a deformation of the sheaf } \Eps \textrm{ over } A \\
f^i_A  & \textrm{is the isomorphism } f^i_A: H^i(\Eps_A) \rightarrow H^i(\Eps)\otimes A \textrm{ for all } i\in \N 
\end{array}  \right\}\right. .
$$

Now let $\Phi$ be the morphism of functors given, for all $A\in Art_{\C}$, by:
$$
\begin{array}{cccl}
\Phi: &  \Def_{\chi}(A) & \longrightarrow &\Def_{A_X^{(0,*)}(\End\Eps)}(A) \\
      & (x,e^a)      & \longrightarrow &  x
\end{array}      
$$
With the above geometric interpretations of the functors $\Def_{\chi}$ and $\Def_{A_X^{(0,*)}(\End\Eps)}(A)$, the morphism $\Phi$ is the one which associates to every pair $(\Eps_A, f^i_A)\in \Def_{\chi}(A)$ the element $\Eps_A\in \Def_{A_X^{(0,*)}(\End\Eps)}(A)$. Thus we have the following characterization of $\Def^0_{\Eps}$ using DGLAs point of view (see \cite{Man sottovar}, Lemma 4.1).
\begin{prop}
The subfunctor $\Def^0_{\Eps}$ is isomorphic to the image of the morphism $\Phi: \Def_{\chi}\rightarrow \Def_{A^{(0,*)}_X(\End\Eps)}$. 
\end{prop}

\section{Proof of the Main Theorem} \label{proof of the main theorem}
This section is devoted to the proof of the Main Theorem. With the above notations
\begin{theorem}[Main Theorem]
The Brill-Noether strata $\mathcal N(h_0\ldots h_n)$ have quadratic algebraic singularities.
\end{theorem}
\begin{proof}
The local study of the strata $\mathcal N(h_0\ldots h_n)$ at one of its point $\Eps$, corresponds to the study of a germ of analytic space which prorepresents the functor $\Def^0_{\Eps}$. 
Our proof is divided into four steps in which we find out a chain of functors, linked each other by smooth morphisms, from the functor $\Def_{\Eps}^0$ to a deformation functor for which it is known that the germ of analytic space that prorepresents it has quadratic algebraic singularities. Then, we  conclude, using properties proved in section \ref{type of singularities}.

\ \\
\emph{First Step.}
We prove that the morphism $\Phi: \Def_{\chi}\rightarrow \Def_{\Eps}^0$ is smooth. 
Then, given a principal extension in $Art_{\C}$, $0\rightarrow J\rightarrow B\stackrel{\alpha}{\rightarrow}A\rightarrow 0$, and an element $(\Eps_A, f^i_A)\in \Def_{\chi}(A)$, we have to prove that, if its image $\Eps_A\in \Def_{\Eps}^0(A)$ has a lifting $\Eps_B\in \Def_{\Eps}^0(B)$, it has a lifting in $\Def_{\chi}(B)$. 

Since $\Eps_A\in \Def_{\Eps}^0(A)$ and $ \Eps_B\in \Def_{\Eps}^0(B)$, their cohomology spaces are deformations of $H^i(\Eps)$ over $A$ and $B$ respectively and so $H^i(\Eps_A)\cong H^i(\Eps)\otimes A$  and $H^i(\Eps_B)\cong H^i(\Eps)\otimes B$. Thus $H^i(\Eps_B)$ is a lifting of $H^i(\Eps_A)$ and it is a polynomial algebra over $B$. 
It follows that in the diagram
$$
\xymatrix{ H^i(\Eps_B) \ar[r]\ar@{-->}[drr]^{f^i_B} & H^i(\Eps_A) \ar[r]^-{f^i_A} &  H^i(\Eps)\otimes A \\      
           B \ar[u]\ar[rr]        &          & H^i(\Eps)\otimes B \ar@{->>}[u]_-{\beta=\Id\otimes\alpha}      } 
$$
there exists a homomorphism $f^i_B: H^i(\Eps_B)\rightarrow H^i(\Eps)\otimes B$, which lifts $f^i_A$. 
Then also the following diagram commutes:
$$\xymatrix@dr{  H^i(\Eps)  &  H^i(\Eps)\otimes B \ar[l]   \\
                 H^i( \Eps_B) \ar[u] \ar[ur]^-{f^i_B}  &  B \ar[u]\ar[l]}$$
and so $f^i_B$ is an isomorphism, for all $i\in \N$.

\ \\
\emph{Second Step.}
Since $X$ is a K\"{a}hler manifold and $\Eps$ is a hermitian sheaf, the operators $\deltabar_\Eps^*$, adjoint of $\deltabar_{\Eps}$, and the Laplacian $\overline{\square}_\Eps=\deltabar_\Eps\deltabar^*_\Eps + \deltabar^*_\Eps \deltabar_\Eps$ can be defined between forms on $X$ with values in the sheaf $\Eps$. 
Let $\H_X^{(0,*)}(\Eps)=\ker \overline{\square}_{\Eps}$ be the complex of $(0,*)$-harmonic forms on $X$ with values in $\Eps$ and let $\Hom^*(\H_X^{(0,*)}(\Eps), \H_X^{(0,*)}(\Eps))$ be the formal DGLA of the homomorphisms of this complex.

Also for the sheaf $\End\Eps$ the operator $\deltabar_{\End\Eps}^*$, adjoint of $\deltabar_{\End\Eps}$, and the Laplacian $\overline{\square}_{\End\Eps}=\deltabar_{\End\Eps}\deltabar^*_{\End\Eps} + \deltabar^*_{\End\Eps} \deltabar_{\End\Eps}$ can be defined. 
Let $\H_X^{0,*}(\End\Eps)=\ker \overline{\square}_{\End\Eps}$ be the complex of the $(0,*)$-harmonic forms on $X$ with values in $\End\Eps$.

Siu proved (see \cite{Siu}) that, for a flat holomorphic vector bundle $\mathcal L$ on a K\"{a}hler manifold $X$, the two Laplacian operators  $\overline{\square}_{\mathcal L}$ and $\square_{\mathcal L}$ coincide. Then a $(0,*)$-form on $X$ with values in $\mathcal L$ is harmonic if and only if it anhilates $\partial$, which is well defined because $\mathcal L$ is flat. 

Since $\End\Eps$ is flat, these facts imply that the complex $\H_X^{0,*}(\End\Eps)$ is a DGLA with bracket given by the wedge product on forms and the composition of endomorphisms.  

Moreover, we can define a morphism $\Omega: \H_X^{(0,*)}(\End\Eps)\rightarrow  \Hom^*(\H_X^{(0,*)}(\Eps), \H_X^{(0,*)}(\Eps) )$. 
Every element $x \in\A_X^{(0,*)}(\End\Eps)$ gives naturally an homomorphism $\Omega(x)$ from $A_X^{(0,*)}(\Eps)$ in itself, defined locally to be the wedge product between forms and the action of the endomorphism on the elements of $\Eps$. If we defined it on an open cover of $X$ on which both the sheaves $\Eps$ and $\End\Eps$ have costant transition functions, when $x \in \H_X^{(0,*)}(\End\Eps)$ and $\Omega(x)$ is restricted to the harmonic forms $\H_X^{(0,*)}(\Eps)$, it gives as a result an harmonic form. 
Let $\Omega: \H_X^{(0,*)}(\End\Eps)\rightarrow  \Hom^*(\H_X^{(0,*)}(\Eps), \H_X^{(0,*)}(\Eps) ) $ be the DGLAs morphism just defined and let  $\Def_{\Omega}$ be the deformation functor associated to it. 

We want to prove that the two functors $\Def_{\Omega}$ and $\Def_{\chi}$ are isomorphic. Then we consider the following commutative diagram: 
$$\xymatrix{ A_X^{(0,*)}(\End\Eps) \ar[d]^-{\chi} & \H^{(0,*)}_X(\End\Eps) \ar[d]^-{\eta} \ar[l]_-{\alpha}\ar[r]^-{\gamma} & \H_X^{(0,*)}(\End\Eps)\ar[d]^-{\Omega} \\
             \Hom^*(A_X^{(0,*)}(\Eps),A_X^{(0,*)}(\Eps)) & M^* \ar[l]_-{\beta}\ar[r]^-{\delta} & \Hom^*(\H_X^{(0,*)}(\Eps),\H_X^{(0,*)}(\Eps))         }$$
where $M^*=\left\{\varphi \in \Hom^*(A_X^{(0,*)}(\Eps),A_X^{(0,*)}(\Eps) )\mid  \varphi (\H_X^{(0,*)}(\Eps))\subseteq \H_X^{(0,*)}(\Eps) \right\}$.\\
The morphism $\beta$ is a quasiisomorphism, infact it is injective and $\coker \beta = \Hom^*(A_X^{(0,*)}(\Eps), A_X^{(0,*)}(\Eps))/ M^*$ $\cong \Hom^*( A_X^{(0,*)}(\Eps), A_X^{(0,*)}(\Eps)/\H_X^{(0,*)}(\Eps) )$ is an acyclic complex. Then $\alpha$ and $\beta$ induce a quasiisomorphism between the cones $C_{\eta}\rightarrow C_\chi$ and so, by the Inverse Function Theorem (Theorem \ref{inverse function theorem}), an isomorphism between the functors $\Def_\eta\rightarrow\Def_\chi$.

Also the morphism $\delta$ is a quasiisomorphism, infact it is surjective and its kernel is $\ker \delta=$ $\left\{\varphi \in \Hom^*(A_X^{(0,*)}(\Eps),  A_X^{(0,*)}(\Eps)) \mid  \varphi(\H_X^{(0,*)}(\Eps))=0  \right\}$ that isomorphic to the acyclic complex $\Hom^*( A_X^{(0,*)}(\Eps)/\H_X^{(0,*)}(\Eps), A_X^{(0,*)}(\Eps))$.
Then $\gamma$ and $\delta$ induce a quasiisomorphism between the cones $C_\eta\rightarrow C_\Omega$ and so an isomorphism between the functors $\Def_{\eta}\rightarrow \Def_{\Omega}$.

\ \\
\emph{Third Step.}
Let $\tilde\H_X(\End\Eps)$ be the DGLA equal to zero in zero degree and equal to $\H_X^{(0,*)}(\End\Eps)$ in positive degrees, with zero differential and bracket given by wedge product on forms and composition of endomorphisms. \\
Let $\tilde\Omega: \tilde\H_X^{(0,*)}(\End\Eps)\rightarrow \Hom^*(\H_X^{(0,*)}(\Eps),\H_X^{(0,*)}(\Eps))$ be the DGLAs morphism defined as in the previous step and let $\Def_{\tilde\Omega}$ be the deformation functor associated to it.\\
The inclusion $\tilde\H_X(\End\Eps) \hookrightarrow \H_X(\End\Eps)$ and the identity on $\Hom(\H_X^{(0,*)}(\Eps),\H_X^{(0,*)}(\Eps))$ induce a morphism of functors $\Psi:\Def_{\tilde\Omega}\rightarrow \Def_{\Omega}$. \\
We note that the morphism induced by $\Psi$ between the cohomology spaces of cones of $\Omega$ and $\tilde\Omega$ respectively is bijective in degree greater equal than $2$ and it is surjective in degree $1$. 
Thus, using the Standard Smootheness Criterion (Theorem \ref{criterio liscezza}), we conclude that $\Psi$ is smooth.

\ \\
\emph{Fourth Step.}
Let's write explicitly $\Def_{\tilde\Omega}$.
The functor $\MC_{\tilde\Omega}$, for all $A\in Art_{\C}$, is given by:
$$ \MC_{\tilde\Omega}(A)= \left\{(x,e^a)\in (L^1\otimes \mathfrak{m}_A) \times \exp(M^0\otimes \mathfrak{m}_A ) \mid dx+\frac{1}{2}[x,x]=0, e^a* \tilde\Omega(x)=0\right\}$$
where $L^*=\tilde\H_X^{(0,*)}(\End\Eps)$ and $M^*=\Hom^*(\H^{(0,*)}_X(\Eps),\H^{(0,*)}_X(\Eps))$. 
Since the differential in the DGLA $\tilde\H_X^{(0,*)}(\End\Eps)$ is zero and since the equation $e^a*\tilde\Omega (x)=0$ can be written as  $\tilde\Omega (x)=e^{-a}*0=0$, we obtain, for all $A\in Art_{\C}$:
$$\MC_{\tilde\Omega}(A)=\left\{ x\in \ker\tilde\Omega\otimes \mathfrak{m}_A \mid [x,x]=0 \right\}\times \exp(\Hom^0(\H^{(0,*)}_X(\Eps),\H^{(0,*)}_X(\Eps)) \otimes \mathfrak{m}_A).$$
Moreover $\exp(\tilde\H_X^{(0,0)}(\End\Eps) \otimes \mathfrak{m}_A )\times \exp(d\Hom^{-1}(\H^{(0,*)}_X(\Eps),\H^{(0,*)}_X(\Eps))\otimes \mathfrak{m}_A)$ is equal to zero, thus there isn't gauge action.  
Thus, for all $A\in Art_{\C}$, we have:
$$\Def_{\tilde\Omega}(A)=\left\{ x\in \ker\tilde\Omega\otimes \mathfrak{m}_A \mid [x,x]=0 \right\}\times \exp(\Hom^0(\H^{(0,*)}_X(\Eps),\H^{(0,*)}_X(\Eps)) \otimes \mathfrak{m}_A).$$ 
Since $\tilde\Omega$ is a DGLAs morphism, $\ker\tilde\Omega$ is a DGLA and it is defined the deformation functor $\Def_{\ker\tilde\Omega}$ associated to it.
Now, for all $A\in Art_{\C}$, we obtain:
$$
\Def_{\tilde\Omega}(A)=\Def_{\ker\tilde\Omega}(A) \times \exp(\Hom^0(\H^{(0,*)}_X(\Eps),\H^{(0,*)}_X(\Eps)) \otimes \mathfrak{m}_A).
$$
The DGLA $\ker\tilde{\Omega}$ has zero differential, so the functor $\Def_{\ker\tilde\Omega}$ is prorepresented by the germ in zero of the quadratic cone (see \cite{Goldman-Millson 2}, Theorem 5.3):
$$ X= \{ x\in \ker^1\tilde\Omega \mid [x,x]=0  \},$$
that has quadratic algebraic singularities. Then also the functor $\Def_{\tilde\Omega}$ is prorepresented by a germ of analytic space with quadratic algebraic singularities.

\ \\
\emph{Conclusion.}
Since now we have constructed smooth morphisms between the functor $\Def_{\Eps}^0$ and the functor $\Def_{\tilde\Omega}$:
$$
\xymatrix{\Def_{\Eps}^0& & \Def_{\chi}\ar[ll]_{\textrm{smooth}}  \ar@{<->}[rr]^{\cong}_{\textrm{isomorphism}} &  &\Def_{\Omega}& &\Def_{\tilde\Omega}\ar[ll]_{\textrm{smooth}} .} $$
By Proposition \ref{prop formalita}, there exists a smooth morphism between the germs of analytic spaces which are hulls of the two functors $\Def_{\Eps}^0$ and $\Def_{\tilde\Omega}$. Moreover, by Theorem \ref{teo sing quadr}, since the germ which is a hull of $\Def_{\tilde\Omega}$ has quadratic algebraic singularities, the same is true for the hull of  $\Def_{\Eps}^0$.
\end{proof}

\end{document}